\documentclass[12pt,a4paper]{amsart}
\usepackage[top=35mm, bottom=35mm, left=30mm, right=30mm]{geometry}
\usepackage{mathptmx}
\usepackage{mathrsfs}
\usepackage{amsmath,amssymb,amscd,verbatim,color}
\usepackage{verbatim}
\usepackage{url}
\usepackage[all]{xy}
\usepackage{xcolor}
\usepackage[colorlinks=true,citecolor=blue]{hyperref}
\usepackage{enumerate}
\usepackage{amsthm}
\usepackage[cp1250]{inputenc}
\usepackage[T1]{fontenc}
\usepackage{latexsym}

\theoremstyle{plain}
\newtheorem*{THA}{Theorem A}
\newtheorem*{THB}{Theorem B}
\newtheorem*{THC}{Theorem C}
\newtheorem{Theorem}{Theorem}[section]
\newtheorem{Lemma}{Lemma}[section]

\newtheorem{Remark}{Remark}[section]

\newcommand{\T}{\mathbb T}

\begin{document}

\title{topological complexity, minimality and systems of order two on torus}

\author[Y. Qiao]{Yixiao Qiao}
\address{Y. Qiao: Department of Mathematics, University of Science and Technology of China, Hefei, Anhui, 230026, P.R.China}
\email{yxqiao@mail.ustc.edu.cn}



\thanks{The author is partially supported by NNSF of China (11225105,11371339)}

\begin{abstract} The dynamical system on $\T^2$ which is a group extension over an irrational
rotation on $\T^1$ is investigated.
The criterion when the extension is minimal, a system of order 2 and when
the maximal equicontinuous factor is the irrational rotation is given.
The topological complexity of the extension is computed, and a  negative answer to the latter part of an
open question raised by Host-Kra-Maass \cite{HKM} is obtained.
\end{abstract}
\maketitle
\section{Introduction}
Throughout this paper, by a {\it topological dynamical system} (t.d.s. for short) we mean a pair $(X, T)$, where $X$ is a
compact metric space and $T: X\rightarrow X$ is a homeomorphism.
In this section, we first discuss the motivations and then state the main results of the article.

The study of the complexity of a dynamical system is one of the main topics in the study
of the system. There are several ways to measure the complexity of a t.d.s.. Entropy is a topological invariant
and a t.d.s. with positive entropy means that the complexity of the system is \textquotedblleft big\textquotedblright. We now discuss
the so-called {\it topological complexity}, which was formally introduced in \cite{BHM} and
is suitable to measure systems with 'lower' complexity, especially systems with zero entropy.
Let $(X,T)$ be a t.d.s. and $\mathcal{U}$ be an open cover of $X$. Define the complexity function with respect to
$\mathcal{U}$ as $n\mapsto \mathcal{N}(\vee_{i=0}^{n-1}T^{-i}\mathcal{U})$. We remark that studying the
topological complexity for a subshift has a long history, which is the complexity with respect to the open
cover consisting of cylinders of length 1, see for instance \cite{MQ}.

It was shown in \cite{BHM} that a t.d.s.
is equicontinuous if and only if each nontrivial open cover has a bounded topological complexity.
Since an equicontinuous system is distal (which has zero topological entropy) and each minimal distal t.d.s. is the result of a transfinite sequence of equicontinuous extensions, and their limits, starting from a t.d.s. consisting of a singleton, it is natural to ask what the complexity of a minimal distal system could be.

For a special
class of minimal distal systems, namely systems of order $d$ which are the inverse limit of minimal $d$-step nilsystems
(see Section 2.4 for definitions) it was proved in \cite{DDMSY} that the complexity function
is bounded  above by a polynomial. In \cite{HKM} the authors refined the result of \cite{DDMSY} by giving
the explicit degree of the polynomial and showing that the lower bound and the upper bound have the same degree.
To state the result we note that the complexity defined by the open cover can be rephrased in the language of
$(n, \epsilon)$-spanning sets, namely one may consider the smallest cardinality $r(n,\epsilon)$
of $(n, \epsilon)$-spanning sets
instead of the smallest cardinality of the subcovers. In this language one of the main results in \cite{HKM}
can be stated as follows:

Let $(X=G/\Gamma, T)$ be a minimal $s$-step nilsystem (see section 2.4 for the definition) for some $s\geq 2$  and assume
that $(X, T)$ is not an $(s-1)$-step nilsystem. Let $d_X$ be a distance on $X$ defining its topology.
Then for every $\epsilon>0$ that is sufficiently small, there exist positive constants $c(\epsilon),
c'(\epsilon)$ and $p\geq s-1$ such that the topological complexity $r(n, \epsilon)$ of $(X, T)$ for the distance $d_X$ satisties
$$c(\epsilon)n^p\leq r(n, \epsilon)\leq c'(\epsilon)n^p  \text{ for every }n\geq 1.$$
Moreover,  $c(\epsilon)\rightarrow + \infty$ as $\epsilon\rightarrow 0$.

An open question asked in \cite[Question 1]{HKM} is what systems have the same topological complexity as nilsystems, namely,

\medskip
\noindent{\bf Question:} {\it Characterize the minimal t.d.s. $(X, T)$ satisfying the following property (\ref{complexity}): 
\begin{equation*}
\text { for every } \epsilon>0 { \;small \;enough, \;} \text { there exist constants } c_1(\epsilon), c_2(\epsilon)>0 \text { such that } 
\end{equation*}
\begin{equation}\label{complexity}
c_1(\epsilon)n\leq r(n, \epsilon)\leq c_2(\epsilon)n \text{ for every } n\geq 1 \text{ and }c_1(\epsilon) \rightarrow \infty \text{ as } \epsilon\rightarrow 0.
\end{equation}
If in addition, we assume that $(X, T)$ is a distal system, then is it a 2-step nilsystem?}
\medskip

We will give a negative answer to the latter part of this question in this paper. To do this, we consider
a t.d.s. on $\T^2$ which is a group extension over an irrational
rotation on $\T^1$. The criterion when the extension is minimal, a system of order 2 and when
the maximal equicontinuous factor is the rotation on $\T^1$ is given.
We note that dynamical systems on $\T^2$ have been studied by many authors,
see for example \cite{ Fra1, Fra2, Fur, EG}.  

To state our results explicitly, let $(\T^2, T)$ be a t.d.s., 
where
$\T^2=\T^1\times \T^1$ with the metric 
$$d((x_1, y_1), (x_2, y_2))=\max\{||x_1-x_2||, ||y_1-y_2||\}, \text{ here } ||r-s||=\min\limits_{m\in \mathbb{Z} }|r-s+m|$$
and
\begin{equation}\label{defineT}
T: \T^2\rightarrow \T^2, (x, y)\mapsto (x+\alpha, f(x)+y),
\end{equation}
with $$f\in \mathcal{F}_l :=\{h:\mathbb{R}\rightarrow \mathbb{R}: h \text{ is continuous on } \mathbb{R},\ h(x+1)-h(x)\equiv l \text{ for all } x\in \mathbb{R}\},$$
$l \in\mathbb{Z}$ and $\alpha \in \mathbb{R}\setminus\mathbb{Q}$.

 Now we state the main results of this paper.
In Theorem A, we compute the  topological complexity for a class of systems $(\T^2, T)$ when the function $f$ satisfies some mild conditions.

\medskip
\begin{THA}
 Let $(\T^2,T)$ be a t.d.s. defined in (\ref{defineT}) such that 
  $f\in \mathcal{F}_l$, $l\not=0$, $\alpha \in \mathbb{R}\setminus\mathbb{Q}$ and $f$
 has a bounded variation on $[0,1]$.  Then (\ref{complexity}) holds.
\end{THA}
\medskip
In Theorem B, we give a characterization of equivalence condition for the system $(\T^2, T)$ to be  order 2.
\medskip
\begin{THB}
 Let $(\T^2,T)$ be a t.d.s. defined in (\ref{defineT}) such that  $f\in \mathcal{F}_l$ and $l\not=0$.
Then the following statements are equivalent:
\begin{enumerate}[$(1)$]
\item $(\T^2, T)$ is a system of order 2.

\item There exist $\varphi\in \mathcal{F}_0$ and $c\in \mathbb{R}$ such that $f(x)=lx+\varphi(x+\alpha)-\varphi(x)+c$ for any $x\in \mathbb{R}$.
\end{enumerate}
\end{THB}

For an irrational number $\alpha$ we may define a number $\nu(\alpha)$ which measures the approximality
of $\alpha$ by rational numbers, see Section 5. We remark that the Lebesgue measure of $\{\alpha\in (0,1): \nu(2\pi\alpha)=0\}$ is 1.
In Theorem C, we give a  minimal distal system  whose topological complexity is low, but it is not a system of order 2.
 Moreover, by the construction of our example, we know that such systems are numerous.
\medskip
\begin{THC}
For a given $l \neq 0$ and an irrational number
$\alpha$ with $v(2\pi\alpha)=0$, there exists a function $f\in \mathcal{F}_l$ such that  
the t.d.s. $(\T^2, T)$ defined in (\ref{defineT}) by $f$
is a minimal distal system but not a system of order $2$, and (\ref{complexity}) holds.
\end{THC}

\medskip
 For readers interested in zero entropy diffeomorphisms on manifolds (particularly $\T^2$),
 it is worth mentioning that, Fr\k{a}czek \cite{Fra1, Fra2} concentrated on ergodic diffeomorphisms of $\T^2$ 
 with polynomial (or linear) growth of the derivative and obtained that they are (in some sense) ``conjugate''  to 
  (\ref{defineT}) with $l\neq 0$.

The paper is organized as follows. In Section 2, some definitions  and related
results are introduced. In Section 3, it is proved that if $f$ satisfies some mild conditions,
then its topological complexity  of $(\T^2, T)$ is low. 
In  Section 4, it is shown that some system $(\T^2, T)$ is minimal distal 
and its maximal equicontinuous factor is the irrational  rotation.
In Section 5, we give the main result in this paper, 
that is, there exists a minimal distal system $(\T^2, T)$
such that its topological complexity is low and it is not a system of order 2.

\noindent{\bf Acknowledgement:} I would like to thank Professors Wen Huang and Xiangdong Ye for their useful suggestions.

\section{Preliminaries}
\subsection{Topological dynamical systems}
A {\it topological dynamical system} (t.d.s. for short) is a pair $(X, T)$, where $X$ is a compact metric space and
$T: X\rightarrow X$ is a homeomorphism from $X$ to itself. We use $d$ to denote the metric on $X$.

A t.d.s. $(X,T)$ is {\it transitive} if for any non-empty open sets $U$ and $V$ in $X$, there exists $n\in \mathbb{Z}$ such that $U \cap T^{n}V\neq \emptyset$. 
We say $x\in X$ is a {\it transitive point} if its orbit orb$(x, T)=\{x, Tx, T^2x, \cdots\}$ is dense in $X$.
A  t.d.s. $(X, T)$  is {\it minimal} if the orbit of any point is dense in $X$.
We say $x\in X$ is a {\it minimal point} if $(\overline{\text{orb}(x, T)}, T)$ is a minimal subsystem of $(X, T)$.

Let $(X, T)$ be a t.d.s.and $(x, y)\in X\times X$. We say that $(x, y)$ is  a {\it proximal pair} if
$$\inf _{n \in \mathbb{Z}}d(T^nx, T^ny)=0,$$
and it is a{ \it distal pair }if it is not proximal. A t.d.s. $(X, T)$ is called {\it distal}
if $(x, y)$ is distal whenever $x ,y\in X$ are distinct. The following result is classical.

\begin{Lemma} \label{Q15}\textnormal{(See \cite{JA})}
Suppose $(X, T)$ is a distal t.d.s., then for any point $x\in X$, $x$ is a minimal point.
In particular, if $(X, T)$ is distal, then $(X, T)$ is minimal if and only if $(X, T)$ has a transitive point.
\end{Lemma}

A {\it homomorphism} $\pi: X\rightarrow Y$ between topological dynamical systems $(X, T)$ and $(Y, S)$ is a continuous onto map such that $\pi\circ T=S\circ \pi$; one says that $(Y, S)$ is a {\it factor} of $(X, T)$ and that $(X, T)$ is an {\it extension} of $(Y, S)$, and one also refers to $\pi$ as a {\it factor map} or an {\it extension}. The systems are said to be {\it conjugate} if $\pi$ is bijective.

Given a t.d.s. $(X, T)$, define the {\it regionally proximal relation}:
$$Q(X, T)=\bigcap_{k=1}^{+\infty}\overline{\bigcup_{n=-\infty}^{+\infty}(T\times T)^{-n}\Delta_{\frac{1}{k}}},$$
where $\Delta_{\frac{1}{k}}:=\{(x,y)\in X\times X:d(x,y)<1/k\}.$
It is clear that  $(x, y)\in Q(X, T)$ if and only if for any $\epsilon >0$, any neighbourhoods $U$ and $V$ of $x$ and $y$ respectively, there exist $x'\in U$, $y'\in V$ and $n\in \mathbb{Z} $ such that $d(T^nx',  T^ny')<\epsilon$.

A t.d.s. $(X, T)$ is said to be {\it equicontinuous} if the family of $\{T^n: n\in \mathbb{Z} \}$ is equicontinuous,
that is , for any $\epsilon >0$, there exists $\delta>0$ such that if $d(x_1, x_2)<\delta$, 
then $d(T^nx_1, T^nx_2)<\epsilon$ for any $n\in \mathbb{Z} $. The following result is well known.

\begin{Lemma}\label{lm1}\textnormal{(See \cite[Chapter 5]{JA})}
Suppose $\pi: (X, T)\rightarrow(Y,S)$ is  an extension between  two t.d.s.. Then the following are equivalent:

\begin{enumerate}[$(1)$]
\item $(Y, S)$ is equicontinuous.

\item $Q(X,T)\subset R_\pi$, where $R_\pi=\{(x, y)\in X\times X: \pi(x)=\pi(y)\}$.
\end{enumerate}
\end{Lemma}
\noindent In particular, the maximal equicontinuous factor of $(X, T)$  is induced by the smallest closed invariant equivalence relation containing $Q(X, T)$.

\subsection{Topological complexity}
Let $(X, T)$ be a t.d.s. and denote by $d$ the metric on $X$.
For any $n\in\mathbb N$   and $\epsilon >0$, a subset $F$ of $X$ is said to be an {\it $(n, \epsilon)$-spanning set of $X$} with respect to $T$ if for any $x\in X$, there exists $y\in F$ with $d_n(x, y)\leq \epsilon$, where
 $$d_n(x, y)=\max\limits_{0\leq i\leq n-1} d(T^i(x), T^i(y)).$$
 Let $r(n, \epsilon)$ denote the smallest cardinality of all $(n, \epsilon)$-spanning subsets of $X$ with respect to $T$, we call $r(n, \epsilon)$ the {\it topological complexity }of the system $(X, T)$. We write $r(n, \epsilon, T)$ to emphasise $T$ if we need to.
We can also define topological complexity in terms of $(n, \epsilon)$-separated set.
A subset $E$ of $X$ is said to be an { \it $(n ,\epsilon)$-separated } set of $X$ with respect to $T$ if $x, y\in E, x\neq y$, implies $d_n(x, y)>\epsilon$, where $d_n(x, y)$ is defined as mentioned above.  Let $s(n, \epsilon)$ denote the largest cardinality of all $(n, \epsilon)$ separated subsets of $X$ with respect to $T$. We write $s(n, \epsilon, T)$ to emphasise $T$ if we need to.

We have $$r(n, \epsilon)\leq s(n, \epsilon)\leq r(n, {\epsilon}/{2})$$
 for any $\epsilon>0$ and $n\in \mathbb{N}$ (see \cite[Page 169]{PW} for details).

\subsection{Unique ergodicity}
Suppose $(X,  \mathcal{B}(X), \mu)$ is a probability space, where $X$ is a compact metrisable space and
$\mathcal{B}(X)$ is the smallest $\sigma$-algebra generated by all open subsets of $X$.

A continuous transformation $T: X\rightarrow X$ is called {\it uniquely ergodic} if there is only one $T$-invariant  Borel probability measure $\mu$ on $X$,
 i.e. $\mu(T^{-1}(B))=\mu(B)$ for all $B\in \mathcal{B}(X)$.

It is well known that if $T(x)=ax$ is a rotation on the compact metrizable group $G$, then $T$ is uniquely ergodic iff $T$ is minimal.
The Haar measure is the only $T$-invariant measure (see for example \cite[Page 162]{PW}).

\begin{Lemma}\label{R}\textnormal{ (See \cite{JCO})}
Let $T: X \rightarrow X$ be a continuous transformation of a compact metrizable space $X$. Then the following are equivalent:
\begin{enumerate}[$(1)$]
\item T is uniquely ergodic.

\item  There exists a $T$-invariant Borel probability measure $\mu$ such that for all $f\in C(X)$ and all $x\in X$,
$$\frac{1}{n}\sum_{i=0}^{n-1}f(T^ix) \rightarrow \int_X fd{\mu}$$ as $n\rightarrow +\infty$.

\end{enumerate}

\end{Lemma}

\subsection{Nilpotent groups, topologically nilpotent groups and systems of order 2 }
Let $G$ be a group. For $g, h\in G$, we write$[g, h]=ghg^{-1}h^{-1}$ for the commutator of $g$ and $h$
and for $A, B\subset G$, we write $[A, B]$ for the subgroup spanned by $\{[a, b]: a\in A, b\in B\}$.
The commutator subgroups $G_j, j\geq 1$, are defined inductively by setting $G_1=G $ and $G_{j+1}=[G_j, G]$.
Let $d\geq 1$ be an integer, we say that $G$ is {\it $d$-step nilpotent }if $G_{d+1}$ is the trivial subgroup.

Let $G$ be a $d$-step nilpotent Lie group and $\Gamma$ a discrete cocompact subgroup of $G$. The compact
manifold $X=G/\Gamma$ is called a {\it $d$-step nilmanifold}. Since $G$ is a nilpotent Lie group,
the commutators subgroups are closed and then, in this case the notions of $d$-step nilpotent and
$d$-step topologically nilpotent coincide (see for example  \cite{AM}). The group $G$ acts on $X$ by left
translations and we write this action as $(g, x)\mapsto gx$. Let $\tau\in G$ and $T$ be the
transformation $x\mapsto \tau x$, then $(X,T)$ is called a {\it $d$-step nilsystem}.

The enveloping semigroup  (or Ellis semigroup) $E(X)$ of a topological dynamical system $(X, T)$
is defined as the closure in $X^X$ of the set $\{T^n: n \in \mathbb{Z} \}$ endowed with the product topology.

Let $(Y, S)$ be a t.d.s., $K$ a compact group, and $\phi: Y\rightarrow K$ a continuous mapping.
Form $X=Y\times K$ and define $T: X\rightarrow X$ by $T(y, k)=(Sy, \phi(y)k)$.
The resulting system $(X, T)$ is called a {\it group extension} of $(Y, S)$.
It is obvious that the system $(\T^2, T)$ defined in (\ref{defineT}) is a group extension of 
an irrational rotation on $\T^1$ by taking $\phi=f$.

The following theorem relates the notion of system of order 2 and nilpotent group which
will be used in this paper. We recall that a  minimal topological
dynamical system is {\it a system of  order d} if it is
an inverse limit of $d$-step nilsystems. In particular, a 2-step nilsystem is a system of  order 2.
\begin{Theorem}\label{lm4}\textnormal{(See \cite[Theorem 1.2]{SD})}
Let $(X, T)$ be a minimal t.d.s.. Then the following are equivalent:
\begin{enumerate}[$(1)$]
\item $(X ,T)$ is a system of order 2.

\item $E(X)$ is a 2-step nilpotent group and $(X, T)$ is a group extension of an equicontinuous system.

\end{enumerate}
\end{Theorem}

The following theorem gives a more explicit characterization for the enveloping semigroup $E(\T^2)$ to become 2-step nilpotent.

\begin{Theorem}\label{lm2}\textnormal{ (See \cite[Theorem 2.3]{EG})}
Suppose  $(\T^2,T)$ is minimal, which is the t.d.s. defined in (\ref{defineT}) such that $f\in \mathcal{F}_l$, $l\not=0$, $\alpha \in \mathbb{R}\setminus\mathbb{Q}$ and that the projection $(\T^2, T)\xrightarrow{\pi} (\T^1, \tau)$ onto the first coordinate is the maximal equicontinuous factor,
 where $\tau: \T^1\rightarrow \T^1, x\mapsto x+\alpha$. Then the following are equivalent:

\begin{enumerate}[$(1)$]
\item There exist $\varphi \in \mathcal{F}_0$ and $c \in \mathbb{R}$ such that $f(x)=\varphi (x+\alpha)-\varphi (x)+lx+c$.

\item The system $(\T^2, T)$ satisfies that $E(\T^2)(\text{as an abstract group})$ is 2-step nilpotent.
\end{enumerate}

\end{Theorem}

\section{Proof of Theorem A}
In this section, the topological complexity of the dynamical system on $\T^2$ is
computed. We will show that  their topological complexity is low in some cases.
Firstly, we introduce some notations. Let $f$ be a real valued function on $[a, b]$,
$$\Delta: a=x_0<x_1<x_2<\cdots <x_n=b$$ be a partition, $$v_\Delta=\sum_{i=1}^{n}|f(x_i)-f(x_{i-1})|,$$
 and $$\bigvee_a^b(f)=\sup\{v_\Delta: \Delta \text{ is a partition over }[a,b]\}.$$ We say that $f$ is a
function with bounded variation if $\bigvee_a^b(f)<+\infty$.

It is well known that $f$ has a  bounded variation on $[a, b]$ if and only if $f(x)=g(x)-h(x) \text{ for all } x\in [a, b]$,
where $g(x)=\frac{1}{2}(\bigvee\limits_{a}^{x}(f)+f(x))$ and $h(x)= \frac{1}{2}(\bigvee\limits_{a}^{x}(f)-f(x))$ are increasing functions on $[a, b]$.

\begin{Lemma}\label{QP1}
Let $(\T^2, T)$ be a t.d.s. defined in (\ref{defineT}) such that  $f\in \mathcal{F}_l, l \neq 0\text{ and }\alpha \in \mathbb{R}\setminus\mathbb{Q}$.
If $f$ has a bounded variation on $[0,1]$, then
\begin{equation}\label{Q1}
s(n,\epsilon)\leq {20(\bigvee\limits_0^1(f)+1)n}/{\epsilon^2}
\end{equation}
for any $n\in \mathbb{N}$ and $\epsilon >0$  small enough.

\end{Lemma}
\begin{proof}
Clearly, $$T^n(x,y)=(x+n\alpha,f_n(x)+y),$$
where $f_n(x)=\sum\limits_{i=0}^{n-1}f(x+i\alpha)$.

Let  $g(x)=\frac{1}{2}(\bigvee\limits_{0}^{x}(f)+f(x))$ and  $h(x)= \frac{1}{2}(\bigvee\limits_{0}^{x}(f)-f(x))$. 
Then $g$ and $h$ are increasing functions on $[0, 1]$ satisfying 
$$f(x)=g(x)-h(x),$$
$$g(x+1)-g(x)= M,$$
 and 
 $$h(x+1)-h(x)=M-l$$
for any $x\in\mathbb{R}$. Take $M=\frac{1}{2}(\bigvee\limits_0^1(f)+l)$, then $M\geq0$ and $M-l\geq0$.

For $n\in \mathbb{N}$,
let $$g_n(x)=\sum\limits_{i=0}^{n-1}g(x+i\alpha) \ \text{ and }\
h_n(x)=\sum\limits_{i=0}^{n-1}h(x+i\alpha).$$
We can choose
a partition $$\Delta_1: 0=s_0<s_1<s_2<\ldots<s_m=1$$ such that
$s_{i+1}-s_{i} \leq {\epsilon}/{2} $ for $0\leq i\leq m-1$ with $m=[{2}/{\epsilon}]+1$,
and a partition  
$$\Delta_2: 0=t_0<t_1<
t_2<\ldots<t_r=1$$ such that 
$$g_n(t_{j+1})-g_n(t_j)\leq {\epsilon}/{4}$$ and
$$h_n(t_{j+1})-h_n(t_j)\leq {\epsilon}/{4}$$ for any $0\leq j\leq r-1$ 
 with $r\leq [{4Mn}/{\epsilon}]+[{4(M-l)n}/{\epsilon}]+1$.

By joining $\Delta_1$ with  $\Delta_2$, we  get a new partition $$0=x_0< x_1<x_2<\ldots < x_p=1$$ for $\epsilon>0$  small enough with $p\leq m+r\leq  {5(\bigvee\limits_0^1(f)+1)n} /{\epsilon}$.

We can also take a partition $$0=y_0<y_1<y_2<\ldots<y_q=1$$
such that $y_{i+1}-y_{i} < {\epsilon}/{3}$ for any $0\leq i\leq q-1$ and $\epsilon>0 $ small enough
 with  $q=[{3}/{\epsilon}]+1\leq {4}/{\epsilon}$.

 Now we show that $s(n,\epsilon)\leq pq\leq {20(\bigvee\limits_0^1(f)+1)n}/{\epsilon^2}$ for $\epsilon>0$ small enough. 
 Suppose this result dose not hold, then there exists an $(n ,\epsilon)$-separated set $E$ of $(\T^2, T)$ such that $|E|>{20(\bigvee\limits_0^1(f)+1)n}/{\epsilon^2}$.
  Since $|E|>pq$, there exist $0\leq i\leq p-1$ and $0\leq j\leq q-1$ such that there are at least two points 
    $$(u_1,v_1), (u_2,v_2)\in E \cap([x_i,x_{i+1}]\times [y_j,y_{j+1}])$$ with $(u_1,v_1)\neq(u_2,v_2)$. Thus,
   \begin{align*}
&d_n((u_1,v_1),(u_2,v_2))\\
=&\max_{0\leq i \leq n-1}d(T^i(u_1,v_1),T^i(u_2,v_2))\\
 =& \max_{0\leq i \leq {n-1}}\max\{||u_1-u_2||, ||g_i(u_1)-g_i(u_2)-h_i(u_1)+h_i(u_2)+v_1-v_2||\} \\
\leq&\max_{0\leq i \leq {n-1}} \max \{||u_1-u_2||, ||g_i(u_1)-g_i(u_2)||+||h_i(u_1)-h_i(u_2)|| +||v_1-v_2||\}\\
\leq&\max_{0\leq i \leq {n-1}}\{|u_1-u_2|, |g_i(u_1)-g_i(u_2)|+|h_i(u_1)-h_i(u_2)| +| v_1-v_2|\}\\
\leq&\max \{|u_1-u_2|, |g_n(u_1)-g_n(u_2)|+|h_n(u_1)-h_n(u_2)| +|v_1-v_2|\}\\
< &\epsilon.
\end{align*}
 This is a contradiction with the definition of the $(n ,\epsilon)$-separated set $E$.
 This shows that \eqref{Q1} holds.
\end{proof}

\begin{Lemma}\label{QP2}
Let $(\T^2, T)$ be a t.d.s. defined in (\ref{defineT}) such that  $f\in \mathcal{F}_l\text{ and } l \neq 0$.
Then
\begin{equation}\label{Q2}
s(n,\epsilon)\geq {n| l | }/{3(\epsilon+\eta (\epsilon))}  
\end{equation}
for  any $n \in \mathbb{N}$ and $\epsilon>0$  small enough, where $\eta(\epsilon)=\sup\limits_{|x-y| \leq \epsilon}|f(x)-f(y)| $.
\end{Lemma}
\begin{proof}
Clearly,  $$T^n(x,y)=(x+n\alpha,f_n(x)+y),$$
where $f_n(x)=\sum\limits_{i=0}^{n-1}f(x+i\alpha)$.
Since  $f_n(1)-f_n(0)=nl$, one has either $$f_n({1}/{2})-f_n(0)\geq {nl}/{2}$$ or $$f_n(1)-f_n({1}/{2})\geq {nl}/{2}.$$
Without loss of generality, we suppose $f_n({1}/{2})-f_n(0)\geq {nl}/{2}$. For the other case, the argument is similar.
Since $f$ is continuous and $f(x+1)-f(x)=l$, it is not hard to check that $\epsilon
\searrow 0$ implies $\eta(\epsilon)\searrow 0$. Take $\epsilon_0>0$ such that $\epsilon_0+\eta(\epsilon_0)<{1}/{3}$.
We can find a sequence $$0=x_0<x_1<x_2\ldots <x_k\leq {1}/{2}$$  such that $$f_n(x_{i+1})-f_n(x_i)= \eta(\epsilon)+\epsilon$$ 
 for any $0\leq i \leq k-1$ with $k> {n|l|}/{3(\eta(\epsilon)+\epsilon)}$ and $\epsilon_0\geq \epsilon>0$.

 Fix $\epsilon\in (0, \epsilon_0]$. 
To show \eqref{Q2} 
it suffices to show that $\{(x_i,0)| 1\leq i \leq k\}$ is an $(n,\epsilon)$-separated set of $(\T^2, T)$.
In fact, for any $1\leq i<j\leq k$, if $x_j-x_i>\epsilon$, then $ ||x_i-x_j||>\epsilon$ which implies that  $d_n((x_i,0),(x_j,0))>\epsilon$.
Otherwise,  by the definition of $\eta(\epsilon)$, we have
$$-\eta(\epsilon)\leq f_1(x_j)-f_1(x_i)\leq \eta(\epsilon),$$
and $$f_{m}(x_j)-f_{m}(x_i)\in [f_{m-1}(x_j)-f_{m-1}(x_i)-\eta(\epsilon),  f_{m-1}(x_j)-f_{m-1}(x_i)+\eta(\epsilon)]$$
for any $2\leq m \leq n$.

Since $f_n(x_j)-f_n(x_i)\geq\eta(\epsilon)+\epsilon$, 
we can define  $$l:=\min\{ 1\leq k \leq n : f_k(x_j)-f_k(x_i)>\epsilon\}.$$  This means $$\epsilon< f_l(x_j)-f_l(x_i) \leq \eta(\epsilon)+\epsilon.$$

Thus,
$$d_n((x_i,0),(x_j,0))\geq d_l((x_i,0),(x_j,0))\geq ||f_l(x_i)-f_l(x_j)|| > \epsilon.$$
Summarizing up, we always have $$d_n((x_i,0),(x_j,0))> \epsilon$$  for any $1\leq i<j\leq k$. This completes the proof.
\end{proof}
Now we turn to prove Theorem A.

\begin{proof}[Proof of Theorem A]
Since $r(n, \epsilon)\leq s(n, \epsilon)\leq r(n, {\epsilon}/{2})$ for any $\epsilon>0$ and $n\in \mathbb{N}$, we have
$$ {n| l | }/{3(2\epsilon+\eta (2\epsilon))} \leq r(n, \epsilon)\leq {20(\bigvee\limits_0^1(f)+1)n}/{\epsilon^2}$$ for $\epsilon>0$ small enough, 
where $\eta(\epsilon)$  comes from Lemma  \ref{QP2}. 
Take $$c_1(\epsilon)= {| l | }/{3(2\epsilon+\eta (2\epsilon))} \; \text{and}\; c_2(\epsilon)={20(\bigvee\limits_0^1(f)+1)}/{\epsilon^2},$$ we get the result. 
\end{proof}

We now  translate Theorem A into the language of topological complexity using open covers. Let $\mathcal{U}$ be an open cover of $X$, and for every integer $n\in \mathbb{N}$,  $\mathcal{N}(\mathcal{U}, n)$ to be the minimal cardinality of a subcover of $\bigvee_{j=0}^{n-1}T^{-j}\mathcal{U}$. The complexity function of $\mathcal{U}$ is the map $n\mapsto \mathcal{N}(\mathcal{U}, n)$.

 We know that $r(n, \epsilon)\leq c(\epsilon)n$ for every $\epsilon>0$ is equivalent to 
 $\mathcal{N}(\mathcal{U},n)\leq C(\mathcal{U})n$ for every open cover $\mathcal{U}$ of $X$; 
 $r(n, \epsilon)\geq c(\epsilon)n$ for every $\epsilon>0$ is equivalent to 
 $\mathcal{N}(\mathcal{U},n)\geq C(\mathcal{U})n$ for every open cover $\mathcal{U}$ of $X$(see \cite{HKM} for details).

\section{Minimality and the maximal equicontinuous factor}

In \cite [Chapter 5] {JA}  the author presented a criterion for the minimality of a class of group extensions
of minimal systems, and applied the criterion to show the minimality of a class of skew products on
$\T^{k+1}$, the $(k+1)$ torus, namely
$$T(z, w_1, w_2, \cdots, w_k)=(\alpha z, \varphi(z)w_1, \varphi(\beta z)w_2, \cdots, \varphi(\beta^{k-1}z)w_k),$$
where $\alpha ,\beta\in \mathbb T^1 $ and $\varphi$ is chosen appropriately.
In this section, we will give another way to prove that when $f \in \mathcal{F}_l(l\neq 0)$,
the system $(\T^2, T)$ defined as before is minimal and the regionally proximal relation
$$Q(\T^2, T)=\{((x, y_1),(x, y_2)): x, y_1, y_2 \in \T^1 \}.$$ For this purpose, the following result is needed and very useful.

\begin {Lemma}\label{Q20}
Let $f\in \mathcal{F}_0$. Then there exist $x_1, x_2 \in \T^1$ such that
\begin{equation}\label{e1}
\sup\limits_{n\geq 1}( f_n(x_1)-n\int_0^1f(x)dx)\leq 2
\end{equation}
and
\begin{equation}\label{e2}
\inf\limits_{n\geq 1}( f_n(x_2)-n\int_0^1f(x)dx)\geq -2,
\end{equation}
where $f_n(x)=\sum\limits_{i=0}^{n-1}f(x+i\alpha)$.
Moreover, the sets $$A=\{x \in \T^1:  \text{there esixts}\ M_1(x)\in \mathbb{R} \text{ such that } \sup\limits_{n\geq1}(g_n(x_1)-n\int_0^1g(x)dx)\leq M_1(x)\}$$
and $$B=\{x\in \T^1:  \text{ there exists } M_2(x)\in \mathbb{R} \text{ such that }  \inf\limits_{n\geq1}(g_n(x_2)-n\int_0^1g(x)dx)\geq M_2(x)\}$$ 
are dense in $\T^1$.
 \end{Lemma}

\begin{proof}
Firstly, we prove \eqref{e1}. Suppose \eqref{e1} dose not hold, then for any $y\in \T^1$, there exists $n_y\geq 1$ such that
$$f_{n_y}(y)-n_y \int_0^1 f(x)dx>2.$$
By the continuity of $f$, there exists an open neighborhood $U_y$  of $y$ such that for any $y'\in U_y$
$$f_{n_y}(y')-n_y \int_0^1f(x)dx>2.$$
Since $\{U_y: y\in \T^1\}$ is an open cover of $\T^1$, there exists $\{y_1, y_2, \cdots, y_l\} \subset \T^1$ such that $\cup_{i=1}^l U_{y_i}=\T^1.$
For $y_0 \in \T^1$,  we define $\{s_i\} \subset \mathbb{N}$  and $\{k_i\} \subset \{1, 2, \cdots, l\}$,  by induction, that $s_0=0, y_0+s_i\alpha \in U_{y_{k_i}}$ and $s_{i+1}=s_i+n_{y_{k_i}}$, then we claim that
\begin{equation}\label{e10}
f_{s_i}(y_0)-s_i\int _0^1f(x)dx\geq 2i \text{ for any } i \geq 1.
\end{equation}

In fact, for $i=1$, it is clear, and if we assume it is true for $i=p$, then it also holds for $i=p+1$ because 
\begin{align*}
f_{s_{p+1}}(y_0)- s_{p+1} \int_0^1f(x)dx=&f_{s_{p}}(y_0)+f_{n_{y_{k_{p}}}}(y_0+s_{p} \alpha)-(s_{p}+n_{y_{k_{p}}})\int_0^1f(x)dx\\
=&(f_{n_{y_{k_p}}}(y_0+s_p \alpha)-n_{y_{k_p}} \int_0^1f(x)dx)+(f_{s_p}(y_0)-s_p \int _0^1f(x)dx)\\
\geq& 2+(f_{s_p}(y_0)-s_p \int _0^1f(x)dx) \; \;\text{(by the definition of $\{s_i\}$)}\\
\geq&2(p+1)  \;\;\text{(by the induction assumption)}.
\end{align*}
Thus, by induction, \eqref{e10} holds.

Let $M=\max\{ n_{y_1}, n_{y_2}, \cdots, n_{y_l}\}$, then $i \leq s_i \leq Mi$ for any $i \geq 1$.
On one hand, we have
\begin{equation}\label{M4}
\limsup_{i \rightarrow +\infty} \frac{f_{s_i}(y_0)}{s_i}
\geq \limsup_{i \rightarrow +\infty}\frac{s_i\int_0^1f(x)dx+2i}{s_i}
\geq\int_0^1f(x)dx+\frac{2}{M}.
\end{equation}
On the other hand, since $\tau: \T^1 \rightarrow \T^1, x\mapsto x+\alpha$ is uniquely ergodic, by Lemma \ref{R}, we have
$$\lim_{n\rightarrow +\infty}\frac{f_n(y_0)}{n}=\int_0^1f(x)dx,$$
a contradiction with \eqref{M4}. This implies that \eqref{e1} holds. 

Now we show that the set $A$ is dense in $\T^1$.  For any $m\in \mathbb{N}$,
\begin{align*}
&\sup_{n\geq1}[g_n(x_1+(m+1)\alpha)-n\int_0^1g(x)dx]\\
=&\sup_{n\geq1}[g_{n+1}(x_1+m\alpha)-g(x_1+m\alpha)-n\int_0^1g(x)dx]\\
=&\sup_{n\geq1}[g_{n+1}(x_1+m\alpha)-(n+1)\int_0^1g(x)dx]+[\int_0^1g(x)dx-g(x_1+m\alpha)].
\end{align*}
So we have $\{x_1+m\alpha:  m\in \mathbb{N}\}\subset A$, hence $A$ is dense in $\T^1$.
By a similar argument, we obtain that \eqref{e2} holds and the set $B$ is dense in $\T^1$.
\end{proof}

\begin{Lemma}\label{Q3}
Let $(\T^2,T)$ be a t.d.s. defined in (\ref{defineT}) such that $f\in \mathcal{F}_l$, $l\neq 0$ and $\alpha \in \mathbb{R}\setminus\mathbb{Q}$. Then $(\T^2, T)$ is minimal.
\end{Lemma}

\begin{proof}
It is clear that $(\T^2, T)$ is distal. To show $(\T^2, T)$ is minimal, by Lemma \ref{Q15},  it suffices to show that  $(\T^2, T)$ is transitive.
Consider non-empty open sets $U_1\times V_1$ and $U_2\times V_2$ of $\T^2$, there exist $x_1, x_2, y_1, y_2\in (0, 1)$ and $\delta>0$ such that
 $$(x_1-\delta, x_1+\delta)\times (y_1-\delta, y_1+\delta)\subset U_1\times V_1$$
 and
 $$(x_2-\delta, x_2+\delta)\times (y_2-\delta, y_2+\delta)\subset U_2\times V_2.$$
Let $f(x)=g(x)+lx$. Since $f\in \mathcal{F}_l$,  we have $g \in \mathcal{F}_0$.

In the following, we divide the proof into two parts.

\noindent\textbf{Case 1}. $l>0$. By Lemma \ref{Q20}, there exist $x'\in(x_1+{\delta}/{4}, x_1+{\delta}/{2}), \;x''\in(x_1-{\delta}/{2}, x_1-{\delta}/{4})$, 
and $M_1,\;M_2\in \mathbb{R}$ such that $$\inf\limits_{n\geq1}(g_n(x')-n\int_0^1g(x)dx)\geq M_1$$ 
and $$\sup\limits_{n\geq1}(g_n(x'')-n\int_0^1g(x)dx)\leq M_2.$$ Thus,
\begin{align}\label{q4}
f_n(x')-f_n(x'')&=g_n(x')-g_n(x'')+nl(x'-x'')\nonumber\\
&\geq  M_1-M_2+{nl\delta}/{2} \rightarrow +\infty
\end{align}
as $n\rightarrow +\infty$.
Since $\{n\alpha| n\in\mathbb{Z^+}\}$ is dense in $\T^1$, there are infinitely many $n_i\in \mathbb{N}$ such that
 $$(n_i\alpha+x_1-{\delta}/{2}, n_i\alpha+x_1+{\delta}/{2})\subset (x_2-\delta, x_2+\delta ).$$ 
  Therefore, for any $x\in (x_1-{\delta}/{2}, x_1+{\delta}/{2})$, we have $x+n_i\alpha\in (x_2-\delta, x_2+\delta)$.
  By \eqref{q4}, there exists $N_1\in \{n_i\}$ such that $$f_{N_1}(x')-f_{N_1}(x'')>1.$$ By the continuity of $f$,
  there exists $x_0\in (x'', x')\subset (x_1-{\delta}/{2}, x_1+{\delta}/{2})$ such that $$f_{N_1}(x_0)+y_1\in
  (y_2-\delta,y_2+\delta) \text { and } x_0+N_1\alpha \in (x_2-\delta,x_2+\delta),$$ i.e.
   $$(x_0, y_1)\in (x_1-\delta, x_1+\delta) \times (y_1-\delta, y_1+\delta)\subset U_1\times V_1$$ and 
   $$T^{N_1}(x_0, y_1)\in (x_2-\delta, x_2+\delta)\times (y_2-\delta, y_2+\delta)\subset U_2\times V_2.$$
  Therefore,  $$(x_0, y_1)\in (U_1\times V_1)\cap T^{-N_1}(U_2\times V_2).$$ 
 Since $U_1, U_2, V_1$ and $V_2$ are arbitrary, $(\T^2,T)$ is transitive.

\noindent\textbf{Case 2}. $l<0$. By Lemma \ref{Q20}, there exist $z'\in(x_1+{\delta}/{4}, x_1+{\delta}/{2}),\;
z''\in(x_1-{\delta}/{2}, x_1-{\delta}/{4})$, and $K_1, K_2\in\mathbb{R}$ such
that $$\sup\limits_{n\geq1}(g_n(z')-n\int_0^1g(x)dx)\leq K_1$$ and
$$\inf\limits_{n\geq1}(g_n(z'')-n\int_0^1g(x)dx)\geq K_2.$$ 
Thus, 
\begin{align}\label{q51}
f_n(z')-f_n(z'')&=g_n(z')-g_n(z'')+nl(z'-z'')\nonumber\\
&\leq  K_1-K_2+nl\delta \rightarrow -\infty
\end{align}
as $n\rightarrow +\infty$.
Since $\{n\alpha| n\in\mathbb{Z}^+\}$ is dense in $\T^1$, there are infinitely many $m_i\in \mathbb{N}$ such that
$$(m_i\alpha+x_1-{\delta}/{2}, m_i\alpha+x_1+{\delta}/{2})\subset (x_2-\delta, x_2+\delta ).$$
 Therefore, for any $x\in (x_1-{\delta}/{2}, x_1+{\delta}/{2})$, we have
$x+m_i\alpha\in (x_2-\delta, x_2+\delta)$. By \eqref{q51}, there exists $N_2\in \{m_i\}$ such that
$$f_{N_2}(x')-f_{N_2}(x'')<-1.$$ 
By the continuity of $f$, there exists $z_0\in (z'', z')\subset
(x_1-{\delta}/{2}, x_1+{\delta}/{2})$ such that
 $$f_{N_2}(z_0)+y_1\in (y_2-\delta,y_2+\delta)\text{ and } z_0+N_2\alpha \in (x_2-\delta,x_2+\delta),$$
  i.e. 
$$(z_0, y_1)\in (x_1-\delta, x_1+\delta)\times(y_1-\delta, y_1+\delta)\subset U_1\times V_1$$ 
and 
$$T^{N_2}(z_0, y_1)\in (x_2-\delta, x_2+\delta)\times(y_2-\delta, y_2+\delta)\subset U_2\times V_2.$$
Therefore, 
 $$(z_0, y_1)\in (U_1\times V_1)\cap T^{-N_2}(U_2\times V_2).$$ 
Since $U_1, U_2, V_1$ and $V_2$ are arbitrary, $(\T^2,T)$ is transitive.
Summarizing up, we complete the proof.
 \end{proof}

\begin{Lemma}\label{Q4}
 Let $(\T^2,T)$ be a t.d.s. defined in (\ref{defineT}) such that  $f\in \mathcal{F}_l$, $l\not=0$ and $\alpha \in \mathbb{R}\setminus\mathbb{Q}$. 
 Then $$Q(\T^2, T)=\{((x, y_1),(x, y_2)):  x, y_1, y_2 \in \T^1\}.$$
\end{Lemma}

\begin{proof}
Firstly, we show that for any $(x_1, y_1), (x_2, y_2)\in \T^2$ with $x_1\neq x_2$,  we have $$((x, y_1),(x, y_2))\notin Q(\T^2, T).$$ 
Fix $x_1,\;x_2, \;y_1\;, y_2\in \T^1$ and let $\epsilon_0={||x_1-x_2||}/{4}$.
Consider  $(x_1-\epsilon_0, x_1+\epsilon_0)\times V_1 \text{ and } (x_2-\epsilon_0, x_2+\epsilon_0)\times V_2, $
 where $V_1, V_2$ are non-empty open neighborhoods of $y_1$ and $y_2$ respectively.
 For  any $(x',y')\in(x_1-\epsilon_0, x_1+\epsilon_0)\times V_1 \text{ and } (x'',y'')\in (x_2-\epsilon_0, x_2+\epsilon_0)\times V_2,$ we have
$$d(T^n(x', y'), T^n(x'', y''))\geq  ||(x'+n\alpha) -(x''+n\alpha)||=||x'- x''|| \geq 2\epsilon_0$$ for each $n\in \mathbb{Z}$.
 So $((x_1, y_1), (x_2, y_2))\notin Q(\T^2, T)$ whenever $ x_1\neq x_2 $.

It remains to  show that  $((x,\; y_1), (x,\; y_2))\in Q(\T^2, T)$ for any $x, \;y_1,\; y_2 \in \T^1$. Fix $x,\; y_1\; y_2 \in \T^1$. 
For any $\epsilon>0$,
suppose $U_1\times V_1$ and $ U_2\times V_2$ are non-empty open neighborhoods of $(x, y_1)$  and
$(x, y_2)$ respectively, then there exists $\delta >0$ ($\delta <\epsilon$) such that
$$(x-\delta, x+\delta)\subset U_1\cap U_2,$$ $$(y_1-\delta,y_1+\delta)\subset V_1 \text{ and } (y_2-\delta,y_2+\delta)\subset V_2.$$

In the following, we divide the proof into two parts.

\noindent\textbf{Case 1}.  
$l>0$. By Lemma \ref{Q20}, there exist $x'\in(x_1+{\delta}/{4}, x_1+{\delta}/{2}),  x''\in(x_1-{\delta}/{2}, x_1-{\delta}/{4})$,
 and $M_1,\; M_2\in\mathbb{R}$ such that 
 $$\inf\limits_{n\geq1}(g_n(x')-n\int_0^1g(x)dx)\geq M_1$$ and $$\sup\limits_{n\geq1}(g_n(x'')-n\int_0^1g(x)dx)\leq M_2.$$ 
 Thus,
\begin{align*}
f_n(x')-f_n(x'')&=g_n(x')-g_n(x'')+nl(x'-x'')\\
&\geq  M_1-M_2+{nl\delta}/{2} \rightarrow +\infty
\end{align*}
as $n\rightarrow +\infty$.
Therefore, there exists $N_1\in \mathbb{N}$ such that $$f_{N_1}(x')-f_{N_1}(x'')>1.$$ By the continuity of $f$, 
there exists $x_0\in (x'', x')\subset (x_1-{\delta}/{2}, x_1+{\delta}/{2})$ 
such that $$||f_{N_1}(x_0)+y_1-f_{N_1}(x')-y_2||<\epsilon,$$ i.e. 
there exist $(x_0,y_1)\in U_1\times V_1$, $(x',y_2)\in U_2\times V_2$ and $N_1\in \mathbb{N}$ such that
  $$d(T^{N_1}(x_0, y_1), T^{N_1}(x', y_2))<\epsilon.$$

 \noindent\textbf{Case 2}. $l<0$. By Lemma \ref{Q20}, there exist  $z'\in(x_1+{\delta}/{4}, x_1+{\delta}/{2}),
z''\in(x_1-{\delta}/{2}, x_1-{\delta}/{4})$ and $K_1, K_2\in \mathbb{R}$ such
that $$\sup\limits_{n\geq1}(g_n(z')-n\int_0^1g(x)dx)\leq K_1$$ and
$$\inf\limits_{n\geq1}(g_n(z'')-n\int_0^1g(x)dx)\geq K_2.$$ 
Thus, 
\begin{align}\label{q5}
f_n(z')-f_n(z'')&=g_n(z')-g_n(z'')+nl(z'-z'')\nonumber\\
&\leq  K_1-K_2+nl\delta \rightarrow -\infty
\end{align}
as $n\rightarrow +\infty$.
 Therefore, there exists $N_2\in \mathbb{N}$ such that $$f_{N_2}(z')-f_{N_2}(z'')<-1.$$ By the continuity of $f$, there exists $z_0\in (z'', z')\subset (x_1-\frac{\delta}{2}, x_1+\frac{\delta}{2})$ such that $$||f_{N_2}(z_0)+y_1-f_{N_2}(z')-y_2||<\epsilon,$$ i.e.  there exist $(z_0,y_1)\in U_1\times V_1$, $(z',y_2)\in U_2\times V_2$ and $N_2 \in\mathbb{N}$ such that
 $$d(T^{N_2}(z_0, y_1), T^{N_2}(z', y_2))<\epsilon.$$
Summarizing up, we finish the proof.
\end{proof}

The following result follows from Lemma \ref{lm1} and Lemma \ref{Q4}.
\begin{Lemma}\label{Q5}
Let $(\T^2,T)$ be a t.d.s. defined in (\ref{defineT}) such that $f\in \mathcal{F}_l$, $l\not=0$ and $ \alpha \in \mathbb{R}\setminus\mathbb{Q}$.
Suppose $$\pi: (\T^2, T)\rightarrow (\T^1, \tau), (x, y)\mapsto x,$$ where $\tau: \T^1\rightarrow \T^1, x\mapsto x+\alpha$.
 Then 
 $(\T^1, \tau)$ is the maximal equicontinuous factor of $(\T^2, T)$.
 \end{Lemma}

Now we prove Theorem B.
\begin{proof}[Proof of Theorem B]
By Lemma \ref{Q3}, $(\T^2, T)$ is minimal. By Lemma \ref{Q5}, we know that 
$(\T^1, \tau)$ is 
the maximal equicontinuous factor of $(\T^2, T)$ where $\tau: \T^1 \rightarrow \T^1, x\mapsto x+\alpha$.
It is clear that $(\T^2, T)$ is an isometric extension of $(\T^1, \tau)$.  We can easily get Theorem B by applying Theorem \ref{lm4} and Theorem \ref{lm2}.
\end{proof}

\section{An example}
In this section, we will give a negative answer to the latter part of the open question raised by Host-Kra-Maass mentioned
in the introduction. That is, we will construct a system whose topological complexity is low
but it is not a system of order 2. Precisely, we will find a bounded variation function $f$
which belongs to $\mathcal{F}_l$ with $l\neq 0$, and at the same time, we define $(\T^2, T)$ in (\ref{defineT})
such that $f$ also satisfies that for any $\varphi \in \mathcal{F}_0$
and $c \in \mathbb{R}$, the equation  $$f(x)=\varphi(x+\alpha)-\varphi (x)+lx+c$$ does not hold.
To do this we start with continued fractions and some related results.

\subsection{Continued fractions}
 A {\it (simple) continued fraction} is a formal expression of the form
 $$a_0+\frac{1}{a_1+\dfrac{1}{a_2+\dfrac{1}{a_3+\dfrac{1}{a_4+_{\ddots}}}}}$$
 which we will also denote by
 $$[a_0; a_1, a_2, ,a_3, \cdots]$$
 with $a_n\in \mathbb{N}$ for $n\geq 1$ and $a_0\in \mathbb{N}_0:=\{0\}\bigcup \mathbb{N}$. The numbers $a_n$ are {\it the partial quotients }of the continued fraction. We also write
$$[a_0; a_1, a_2, \cdots, a_n]$$
for the finite fraction
$$a_0+\frac{1}{a_1+\dfrac{1}{a_2+_{ \ddots+\dfrac{1}{a_{n-1}+\dfrac{1}{a_n}}}}}.$$

We state some basic properties about continued fractions for convenience (See for example \cite{METW} for details):

\begin{enumerate}[$(1)$]
\item The infinite continued fraction converges to a real number, namely, there exists a real number $\alpha$ such that
$$\alpha=[a_0; a_1, a_2, \cdots ]=\lim_{n\rightarrow \infty}[a_0; a_1, a_2, \cdots , a_n].$$
We say that $[a_0; a_1, a_2, \cdots ]$ is  the continued fraction expansion for $\alpha$.

\item Let $a_n\in \mathbb{N}$ for all $n\geq 0$. Then $[a_0; a_1, a_2, \cdots ]$ is irrational.

\item The map that sends the sequence
$$(a_0, a_1, a_2, \cdots )\in \mathbb{N}_0\times \mathbb{N}^\mathbb{N}$$
to $[a_0; a_1, a_2, \cdots ]$ is injective.

\item For any irrational number $\alpha\in (0, 1)$, there exists a continued fraction expansion for $\alpha$.
\end{enumerate}

A real number $\alpha=[a_0; a_1, a_2, \cdots ]\in (0, 1)$ is called{ \it badly approximable }if there is
some $M$ such that $a_n\leq M$ for all $n\geq 1$.
The following result is well known (See for example \cite[Page 87]{METW}).
\begin{Lemma}
A real number $\alpha\in (0, 1)$  is badly approximable if and only if  there exists some constant $c=c(\alpha)>0$ such that
$$ |\alpha- \frac{p}{q}|> \frac{c}{q^2}$$
for every rational number $\frac{p}{q}$.

\end{Lemma}

We define $v(\alpha)=\liminf\limits_{n\rightarrow +\infty } n\| n\alpha\|$. It is clear that $v(\alpha)>0$ if and only if $\alpha$ is badly approximable. It is well known that the set of all badly approximable numbers in $(0, 1)$ is a null set with respect to Lebesgue measure (see for example \cite[Page 87]{METW}). Hence the Lebesgue measure of the set $\{\alpha \in (0, 1): v(\alpha)=0\} $ is one.

\medskip

Now we prove Theorem C.
\begin{proof}[Proof of Theorem C]
 By the definition of $v(2\pi\alpha)$, we know that $\liminf\limits_{n\rightarrow +\infty}n |e^{2\pi in \alpha}-1| =0$.
So there exists an increasing  sequence $\{n_k\}_{k=1}^{+\infty}$ of positive integers such that 
$$n_k| e^{2\pi in_k \alpha}-1|<{1}/{k^2}$$ for every $k\in \mathbb{N}$. Take a function
$$f(x)=lx+\sum_{n=-\infty}^{+\infty}a_ne^{2\pi inx},$$
where
\begin{equation}
a_n=
\begin{cases}
 e^{2\pi i n_k \alpha}-1& \text{if $n=\pm n_k, k\in \mathbb{N}$},\nonumber\\
0& \text{else}.
\end{cases}
\end{equation}
Since $f(x)=\overline{f}(x)$, $f$ is a real valued function.
Since $$f'(x)=l+\sum_{n=-\infty}^{+\infty}2 \pi in a_ne^{2\pi i n x}$$ and
$$|f'(x)|\leq 4\pi \sum_{k=1}^{+\infty} {1}/{k^2}+|l|,$$ we know that $f$ is a continuous function with
 a bounded variation. By Theorem A, we know that (\ref{complexity}) holds.
 
 By the construction of $f$ and Lemma \ref{Q3}, we know that $(\T^2, T)$ is minimal. It is clear that $(\T^2, T)$ is distal.

Next we show that for the function $f$ defined above, the system $(\T^2, T)$ is not a system of order 2. Suppose $(\T^2, T)$ is a system of order 2, by Theorem B, we can assume that there exists  $\varphi\in \mathcal{F}_0$  and $c\in \mathbb{R}$ such that
 \begin{equation}
 f(x)=\varphi(x+\alpha)-\varphi(x)+lx+c\nonumber
 \end{equation}
 for any $x\in \mathbb{R}$.  

Let  $\varphi(x)=\sum\limits_{n=-\infty}^{+\infty} b_n e^{2\pi inx}$ be  the Fourier series of periodic function $\varphi$.
 Comparing the Fourier coefficients of the equation $f(x)-lx=\varphi(x+\alpha)-\varphi(x)+c$, we have
 \begin{equation}
a_n=
\begin{cases}
b_n(e^{2\pi i n\alpha}-1), & \text{ $n \neq 0$},\nonumber\\
c,        & \text{ $n=0$ }.
\end{cases}
\end{equation}
This implies that $\sum\limits_{n=-\infty}^{+\infty}| b_n| ^2=+\infty$, a contradiction with $\sum\limits_{n=-\infty}^{+\infty}| b_n| ^2=\int_0^1 |\varphi(x)|^2 dx<+\infty$. Thus, by Theorem B, we
conclude that $(\T^2, T)$ is not a system of order 2.
\end{proof}

\begin{Remark}
Let $m$ be the  Lebesgue measure on $\mathbb{R}$, $$A=\{\alpha \in \mathbb{R}\setminus\mathbb{Q}: v(\alpha)=0\}$$ and 
$$B=\{\alpha \in \mathbb{R}\setminus\mathbb{Q}: v(2\pi\alpha)=0\}.$$
 Since $m(\{\alpha \in(0, 1): v(\alpha)=0\})=1$, we have $$m(A\cap (0, 2\pi))=2\pi,$$ which implies that $$m(\frac{A}{2\pi}\cap (0, 1))=1.$$ 
 That is, $m(B\cap (0, 1))=1$.
 Therefore for almost all $\alpha \in (0, 1)$ in the sense of Lebesgue measure, there exists $f\in \mathcal{F}_l$ such that Theorem C holds for  the system $(\T^2, T)$.
\end{Remark}

\end{document}